\documentclass[11pt,epsfig]{article}
\usepackage{amsmath}
\usepackage{amsthm}
\usepackage{epsfig}
\usepackage{graphicx}
\usepackage{hyperref}
\usepackage{amsfonts}
\usepackage{graphicx, color, epstopdf}

\newtheoremstyle{thmm}{1.5ex plus 1ex minus .2ex}{1.5ex plus 1ex minus .2ex}{\rmfamily}{}{\bfseries}{}{1em}{}
\theoremstyle{thmm}
\newtheorem{theorem}{Theorem}[section]
\newtheorem{lemma}{Lemma}[section]

\newtheorem{remark}{Remark}
\renewenvironment{proof}[1][Proof]{\noindent\textit{#1. } }{\hfill$\square$}
\def \endproof{\vrule height8pt width 5pt depth 0pt}
\def\refe#1{(\ref{#1})}

\def\R{\mathbb{R}}

\title{\bf Analysis of $L1$-Galerkin FEMs for time-fractional
nonlinear parabolic problems}
\author{Dongfang Li\thanks{School of Mathematics and Statistics,
Huazhong University of Science and Technology, Wuhan 430074, China
({\tt dfli@hust.edu.cn}).
; and Department of Mathematics, City University of Hong Kong, Kowloon,
Hong Kong
}
\and Honglin Liao\thanks{Institute of Sciences,
PLA University of Science and Technology, Nanjing, 211101, China
({\tt liaohl2003@sina.com}).}
\and Weiwei Sun\thanks{Department of Mathematics,
City University of Hong Kong, Kowloon, Hong Kong
({\tt maweiw@math.cityu.edu.hk}).}
\and Jilu Wang\thanks{Department of Scientific Computing,
Florida State University, Tallahassee, FL 32306, USA
({\tt jwang13@fsu.edu}).}
\and Jiwei Zhang\thanks{Beijing Computational Science Research Center,
Beijing  10084, China
({\tt jwzhang@csrc.ac.cn}).}}
\date{}
\begin{document}

\maketitle

\begin{abstract}
This paper is concerned with numerical solutions of
time-fractional nonlinear parabolic problems by a class of $L1$-Galerkin finite element methods.
The analysis of $L1$ methods for time-fractional nonlinear problems
is limited
mainly due to the lack of a fundamental Gronwall type inequality.
In this paper, we establish such a fundamental inequality for
the $L1$ approximation to the
Caputo fractional derivative. In terms of
the Gronwall type inequality, we provide optimal error estimates of
several fully discrete linearized Galerkin finite element methods
for nonlinear problems.
The theoretical results are illustrated by applying our proposed methods
to three examples: linear Fokker-Planck equation, nonlinear
Huxley equation and Fisher equation.



\end{abstract}

\section{Introduction}\label{intro}
\setcounter{equation}{0}
In this paper, we study numerical solutions of
the time-fractional nonlinear parabolic equation
\begin{align}
~_0^C\!\mathcal{D}_t^\alpha u-\Delta u=f(u, x,t),
\quad
x\in\Omega\times(0,T]
\label{i1}
\end{align}
with the initial and boundary conditions, given by
\begin{align}
\begin{array}{ll}
u(x,0)=u_0(x), & x\in\Omega, \\
u(x,t)=0, & x\in\partial\Omega\times[0,T],
\end{array}
\label{i1-ib}
\end{align}
where $\Omega\subset \R^d$ ($d=1,~2$ or $3$) is a bounded
and convex polygon.
The Caputo fractional
derivative $ ~_0^C\!\mathcal{D}_t^\alpha $ is defined as
\begin{equation} \label{defin-1}
 ~_0^C\!\mathcal{D}_t^\alpha u(x,t)
=\frac{1}{\Gamma(1-\alpha)}\int_0^t\frac{\partial
u(x,s)}{\partial s}\frac{1}{(t-s)^\alpha}ds, \quad 0<\alpha <1.
\end{equation}
Here $\Gamma(\cdot)$ denotes the usual gamma function.

The model \eqref{i1} is used to
describe plenty of nature phenomena in physics, biology and chemistry
 \cite{hil00,kil06,mag06,pod99}. In the past decades,
developing effective numerical methods and
rigorous numerical analysis for the time-fractional PDEs have been
a hot research spot \cite{che-liu12,gao15,lan05,mcl09,WH1,WH2,yus05,zeng-li13,zhu09}.
Numerical methods can be roughly divided into two categories:
indirect and direct methods.
The former is based on the solution of an integro-differential equation
by some proper numerical schemes since time-fractional differential equations
can be reformulated into integro-differential equations in general,
while the latter is based on a direct
(such as piecewise polynomial) approximation to the time-fractional
derivative \cite{cao-xu13,cao15,li-tao09,li-chen16}.

Direct methods are more popular in practical computations
due to its ease of implementation.
One of the most commonly used direct methods is
the so-called $L1$-scheme, which can be viewed as a piecewise
linear approximation to the fractional derivative \cite{old74} and
which has been widely applied for solving various time-fractional
PDEs \cite{gao11,jia11}.
However, numerical analysis for direct methods is limited, even
for a simple linear model \refe{i1} with
\begin{align}\label{1.4}
f(u)=L_0 u,~~~~~~t\in (0,T].
\end{align}
The analysis of $L1$-type methods for the linear model
was studied by several authors, while the convergence
and error estimates were obtained under the assumption that
\begin{align}
L_0 \le 0
\label{L-1}
\end{align}
in general, see \cite{jin15,jin-la16,lin07,mart16,sun06}.
The proof there cannot be directly extended to the case of $L_0>0$.
Recently, the condition \refe{L-1}
was improved in \cite{yu-deng15},
in which a time-fractional nonlinear predator-prey model was studied
by an $L1$ finite difference scheme and $f(u)$ was assumed to satisfy
a global Lipschitz condition.
The stability and convergence were proved under the assumption
\begin{align}
T^\alpha< \frac{1}{L \Gamma(1-\alpha)}.
\label{deng}
\end{align}
Here $L$ denotes the Lipschitz constant. The restriction condition \eqref{deng} implies that the scheme is convergent and stable
only locally in time. Similar assumptions appeared in the analysis
of $L1$ type schemes for time-fractional Burger equation
\cite{li-zhang-ran16} and nonlinear Fisher equation
\cite{li-zhang16}, respectively, where $L$ may depend upon an upper bound
of numerical solutions.
In both \cite{li-zhang16} and \cite{li-zhang-ran16}, a classical
finite difference approximation was used for spatial discretization.
Several linearized $L1$ schemes with other approximations
in spatial direction, such as spectral methods
\cite{bhr15,bhr151} and meshless methods \cite{moh13}, were also
investigated numerically for time-fractional nonlinear differential
equations. No analysis was explored there.

It is well known that the classical Gronwall inequality plays
an important role in analysis of parabolic PDEs
($\alpha =1$) and the analysis of corresponding
numerical methods also relies heavily on the discrete counterpart of
the inequality. Clearly, the analysis of
$L1$-type numerical methods for time-fractional nonlinear differential
equations ($0 < \alpha < 1$) has not been well done
mainly due to the lack of such a fundamental inequality.

The aim of this paper is to present the numerical analysis
for several fully discrete $L1$ Galerkin FEMs for the general nonlinear
equation \refe{i1} with any given $T>0$.
The key to our analysis is to establish
a new Gronwall type inequality for a positive
sequence satisfying
\begin{align}
D_{\tau}^{\alpha} \omega^n \le \lambda_1 \omega^n +\lambda_2 \omega^{n-1}+ g^n,
\end{align}
where $D_{\tau}^{\alpha}$ denotes an $L1$ approximation to
$~_0^C\!\mathcal{D}_t^\alpha$,  $\lambda_1$ and $\lambda_2$
are both positive constants. In terms of the fundamental inequality,
we present optimal error estimates of proposed
fully discrete $L1$-Galerkin FEMs for
equation \refe{i1} with linear or nonlinear source $f(u)$.
Moreover, our analysis can be extended to many other direct numerical methods
for time-fractional parabolic equations.

The rest of the paper is organized as follows. We
present three linearized fully discrete numerical schemes
and the main convergence results in Section \ref{sec2}. These schemes are based on
an $L1$ approximation in temporal direction and
Galerkin FEMs in spatial direction.  In Section \ref{sec3},
a new Gronwall type inequality is established for the $L1$ approximation
and optimal error estimates
of the proposed numerical methods are proved.
In Section \ref{sec4},
we present numerical experiments on three different models,
linear fractional Fokker-Planck equation and nonlinear fractional
Huley equation and Fisher equation. Numerical examples are provided to confirm our
theoretical analysis.
Finally, conclusions and
discussions are summarized in Section \ref{sec5}.

\section{ L1-Galerkin FEMs and main results}\label{sec2}
\setcounter{equation}{0}
We first introduce some notations and present several fully discrete
numerical schemes.

For any integer $m\geq 0$ and $1\leq p\leq \infty$,
let $W^{m,p}$ be the usual Sobolev space of functions
defined in $\Omega$ equipped with the norm
$\|\cdot\|_{W^{m,p}}$.
If $p=2$, we denote $W^{m,2}(\Omega)$ by $H^m(\Omega)$.
Let $\mathcal{T}_h$ be a quasiuniform
partition of $\Omega$ into
intervals $T_i$ ($i=1,\cdots,M$) in $\mathbb{R}^1$,
or triangles
in $\mathbb{R}^2$ or tetrahedra in $\mathbb{R}^3$,
 $h=\max_{1\leq i \leq M}\{\textrm{diam}\;T_i\}$
be the mesh size. Let $V_h$ be the finite-dimensional subspace
of $H_0^1(\Omega)$, which consists of continuous
piecewise polynomials of degree $r$ ($r\geq 1$) on $\mathcal{T}_h$.
Let $\mathcal{T}_\tau = \{t_n | t_n=n\tau;0\leq n\leq N \}$ be a uniform
partition of $[0,T]$ with the time step $\tau=T/N$.

Based on a piecewise linear interpolation, the $L1$-approximation (scheme) to the Caputo fractional derivative is given by
\begin{eqnarray} \label{l1sch}
~_0^C\!\mathcal{D}_{t_n}^\alpha u
&=&
\frac{1}{\Gamma(1\!-\!\alpha)}
\int_0^{t_n}\frac{u'(x,s)}{(t_n\!-\!s)^\alpha}ds
\nonumber\\
&=&
\frac{1}{\Gamma(1-\alpha)}\sum_{j=1}^{n}\frac{u(x,t_j)-u(x,t_{j-1})}{\tau}
\int_{t_{j-1}}^{t_{j}}\frac{1}{(t_n-s)^\alpha}ds+ Q^n\nonumber\\
&=&
\frac{\tau^{-\alpha}}{\Gamma(2-\alpha)}\sum_{j=1}^na_{n-j}(u(x,t_j)-u(x,t_{j-1}))
+ Q^n,\nonumber
\end{eqnarray}
where
\begin{align}
a_i=(i+1)^{1-\alpha} - i^{1-\alpha},\quad i\geq 0.
\label{a}
\end{align}
If $u \in
C^2([0,T];L^2(\Omega))$, the truncation error $Q^n$ satisfies
\cite{lin07,sun06}
\begin{eqnarray}\label{2.1}
\|Q^n\|_{L^2} \leq C \tau^{2-\alpha}.
\label{Qn}
\end{eqnarray}
If  $u$ does not have the requisite
regularity, the truncation error $Q^n$ may have some possible
loss of accuracy. We will discuss it later.

For a sequence of functions $\{\omega^n\}_{n=0}^N$, we define
\begin{align}
D_{\tau}^{\alpha}\omega^n
:= \frac{\tau^{-\alpha}}{\Gamma(2-\alpha)}
\sum_{j=1}^na_{n-j}{\delta_t\omega^j}
=  \frac{\tau^{-\alpha}}{\Gamma(2-\alpha)}
\sum_{j=0}^n
 b_{n-j} \omega^j, \quad n=1,\cdots,N,
\label{l1}
\end{align}
where $\delta_t\omega^n=\omega^n-\omega^{n-1}$ and
\begin{align}
\label{add-l1}
b_0 = a_0,\ b_n = -a_{n-1},\ b_{n-j} = a_{n-j} - a_{n-j-1}, \;\;
j = 1,\cdots, n-1.
\end{align}

With above notations, a linearized $L1$-Galerkin FEM is: to seek
$U_h^n\in V_h$ such that
\begin{align} \label{LFEMS}
\left(D_{\tau}^{\alpha}U_h^n, v_h \right)+\left(\nabla  U_h^{n },\nabla v_h \right) =
\left(f\left(U_h^{n-1}\right),v_h \right),   \quad \forall v_h\in V_h, ~~n=1,2,\cdots,N
\end{align}
with $U_h^0=\Pi_hu_0$, where $\Pi_h$ represents the interpolation operator.

By noting \refe{add-l1}, we can rewrite the scheme \eqref{LFEMS} equivalently as
\begin{align}
\frac{\tau^{-\alpha}}{\Gamma(2-\alpha)}
\sum _{j=0}^n b_{n-j} \left (U_h^j, v_h \right )
+ \left(\nabla  U_h^{n },\nabla v_h \right)
= \left(f(U_h^{n-1}), v_h \right), \quad \forall v_h\in V_h.
\label{s1}
\end{align}

In this paper, we assume that the function $f:\mathbb{R}\rightarrow\mathbb{R}$
is Lipschitz continuous, i.e.
\begin{align}
|f(\xi_1)-f(\xi_2)|\leq L|\xi_1-\xi_2|,~~ \quad \textrm{for }\quad \xi_1,\xi_2\in\mathbb{R},
\label{LC}
\end{align}
where $L$ denotes the Lipschitz coefficient. We present optimal
error estimates
of scheme \eqref{s1}
in the following theorem and leave the proof to section \ref{TH2.1Proof}.

\begin{theorem} \label{main}
Suppose that the system \eqref{i1}-\eqref{i1-ib} has a unique solution
$u \in
C^2([0,T];L^2(\Omega))\cap C^1([0,T];H^{r+1}(\Omega))$. Then, there exists a positive constant $\tau_0$,
such that when $\tau\leq \tau_0$,
the finite element system
\eqref{s1}
admits a unique solution $U_h^n$,
$n=1,2,\cdots,N$, satisfying
\begin{eqnarray}
\|u^{n}-U_h^n\|_{L^2} \leq C_0(\tau+h^{r+1}),
\label{error-1}
\end{eqnarray}
where $u^n=u(x,t_n)$ and $C_0$ is a positive constant independent of $\tau$ and $h$.
\end{theorem}

\begin{remark}
We point out that the smoothness of the initial solution and $f$
does not always imply the smoothness of the exact solution for time-fractional equations.
In other words, the exact solution may
not have the requisite regularity around $t=0$ \cite{jin15,jin-la16,mart16}, which may lead to some possible loss of accuracy for $Q^n$.
For example, by taking into account of the possible initial layer and
weaker regularity of the exact solutions (\cite{mart16}, Lemma 5.1),
the maximum truncation error $Q^n$ satisfies
\[\max_{1\leq n \leq N} \|Q^n\|_{L^2}\leq C\tau^\alpha. \]
Then in Theorem \ref{main}, we only can obtain the error estimate by
\begin{eqnarray}
\max_{1\leq n \leq N} \|u^{n}-U_h^n\|_{L^2}
\leq C(\tau^\alpha+h^{r+1}).
\label{error-11}
\end{eqnarray}
The result \eqref{error-11} can be proved similarly without
any additional difficulty
by using our Gronwall type inequality for discrete $L1$-approximation.
\end{remark}

\begin{remark}
The proof of Theorem \ref{main} is based on a Lipschitz condition. If $f\in C^1(\mathbb{R})$, Theorem \ref{main} still holds. In fact, by using the mathematical induction and inverse inequality, we have
\begin{eqnarray}
\| U^{n-1}_h\|_{L^{\infty}}
\leq \| R_hu^{n-1}\|_{L^{\infty}}+\| R_hu^{n-1}-U^{n-1}_h\|_{L^{\infty}}
\leq \| R_hu^{n-1}\|_{L^{\infty}}+ Ch^{-\frac{d}{2}} (\tau + h^{r+1}),
\label{inverse}
\end{eqnarray}where $R_h$ denotes the Ritz projection operator. As
we can see from \eqref{inverse}, the boundedness of $\|U_h^{n-1}\|_{L^\infty}$
can be obtained while mesh size being small. Therefore, we have
 \[
 \|f(u^{n-1})-f(U^{n-1}_h)\|_{L^2}
 = \|f'(\xi) ( u^{n-1}-U^{n-1}_h)\|_{L^2}
 \leq
 C\| u^{n-1}-U^{n-1}_h\|_{L^2},~~\xi \in (u^{n-1},  U_h^{n-1}).
\]
Hence, the results in Theorem \ref{main} can be proved by using similar
analysis under the assumption $f\in C^1(\mathbb{R})$.
\end{remark}

We now present two more high-order fully discrete linearized methods.

With the Newton linearized approximation to the nonlinear term, a
linearized $L1$-Galerkin FEM is: to seek $U_h^n\in V_h$ such that
\begin{eqnarray}
 \left(D_{\tau}^{\alpha}U_h^n, v_h \right)\!+\!\left(\nabla  U_h^{n },
\nabla v_h \right)\! =\!
\left(f(U_h^{n\!-\!1})\!+\!f_1(U_h^{n\!-\!1})
( U_h^n\!-\!U_h^{n\!-\!1}),v_h \right),
    ~~n=1,\cdots,N,\label{s2}
\end{eqnarray}
where $f_1(U_h^{n-1})=\frac{\partial f}{\partial u}|_{u=U_h^{n-1}}$.

Moreover, with an extrapolation to the nonlinear term,
a linearized $L1$-Galerkin FEM is:  to seek $U_h^n\in V_h$ such that
\begin{eqnarray}
\left(D_{\tau}^{\alpha}U_h^n, v_h \right)+\left(\nabla  U_h^{n },\nabla v_h \right) =
\left(f(\widehat U_h^n),v_h \right), ~n=1,\cdots,N,
\label{s3}
\end{eqnarray}
where $\widehat U_h^n=2U_h^{n-1}-U_h^{n-2}$ for $n=2,\cdots,N$ and
$\widehat U_h^1$ can be obtained by solving the governing equation \begin{eqnarray*}
\left(D_{\tau}^{\alpha} \widehat U_h^1, v_h \right)\!
+\!\left(\nabla  \widehat U_h^1,\nabla v_h \right)\! =\!
\left(f( U_h^0)\!+\!f_1( U_h^0)( \widehat U_h^1-\!U_h^0),v_h \right).
\end{eqnarray*}

We next present the error estimates
of schemes \eqref{s2} and \eqref{s3}
in the following theorem.

\begin{theorem} \label{main-1}
Suppose that the system \eqref{i1}-\eqref{i1-ib} has a unique solution
$u \in
C^2([0,T];L^2(\Omega))\cap C^1([0,T];H^{r+1}(\Omega))$. Then, there exists a positive constant $\tau^*_0$,
such that when $\tau\leq \tau^*_0$,
the finite element system
\eqref{s2} or \eqref{s3}
admits a unique solution $U_h^n$,
$n=1,2,\cdots,N$, satisfying
\begin{eqnarray}
\|u^{n}-U_h^n\|_{L^2}
\leq C^*_0(\tau^{2-\alpha}+h^{r+1}),
\label{error-2}
\end{eqnarray}
where $C^*_0$ is a positive constant independent of $\tau$ and $h$.
\end{theorem}

The representation of this paper focuses on the numerical analysis for the linearized scheme \eqref{s1}.
The analysis for \eqref{s1} can be easily extended to the linearized schemes
\eqref{s2} and \eqref{s3}. The main difference is that the schemes \eqref{s2} and \eqref{s3} have the convergent
order $2-\alpha$ in the temporal direction, while
the scheme \eqref{s1} has the order $1$.

In the remainder, we denote by $C$ a generic
positive constant, which is independent of $n,h,\tau,\tau_0,$
$\tau_0^*, C_0$ and
$C_0^*$, and may depend upon $u$ and $f$.

\section{Error analysis}\label{sec3}
\setcounter{equation}{0}
In this section, we will prove the optimal error estimate
given in Theorem 2.1 for proposed scheme \eqref{s1}.
As we can see below, the following Gronwall type inequality plays a key role in our analysis. For brevity, we first present the results of the Gronwall type inequality, and leave the proof to section 3.2.

\begin{lemma}\label{lemma:Gronwall-uniformInitial}
Suppose that the nonnegative sequences $\{\omega^{n}, g^n\,|\,n=0,1,2,\cdots\}$
satisfy
\begin{align*}
D_{\tau}^{\alpha}\omega^n\leq \lambda_1 \omega^n+\lambda_2 \omega^{n-1}+g^n,
\quad n\geq1,
\end{align*}
where $\lambda_1\geq0$ and $\lambda_2\ge0$ are constants.
Then, there exists a positive
constant $\tau^*$ such that, when $\tau\leq \tau^*$,
\begin{align}\label{gron-101}
\omega^n
\leq
2\Big({\omega^0+\frac{t_n^{\alpha}}{\Gamma(1+\alpha)}
\max_{0\leq j\leq n}g^j}\Big)E_\alpha(2\lambda t_n^\alpha),~~~
1\leq n\leq N,
\end{align}
where $E_{\alpha}(z)=\sum_{k=0}^{\infty}\frac{z^{k}}{\Gamma(1+k\alpha)}$ is
the Mittag-Leffler function and
$\lambda=\lambda_1+\frac{\lambda_2}{(2-2^{1-\alpha}) } $.
\end{lemma}

\subsection{Proof of Theorem \ref{main}} \label{TH2.1Proof}
To prove the main results, we first rewrite the system \refe{s1} as
\begin{align}
\frac{\tau^{-\alpha}}{\Gamma(2-\alpha)}
b_0 \left (U_h^n, v_h \right )
+ \left(\nabla  U_h^{n },\nabla v_h \right)
= \left(f(U_h^{n-1}),v_h \right)
- \frac{\tau^{-\alpha}}{\Gamma(2-\alpha)}
\sum _{j=0}^{n-1} b_{n-j} \left (U_h^j, v_h \right ). \label{CM}
\end{align}
It is obvious that the coefficient matrix of the linear system \eqref{CM} is symmetric and positive definite. Thus, the existence and uniqueness of the solution of the FEM system \refe{s1} follow
immediately.

We now let $\Pi_h$ be a Lagrange interpolation operator and
$R_h: H_0^1(\Omega)\rightarrow V_h$ be the Ritz
projection operator defined by
\begin{eqnarray} \label{RPg}
\left(\nabla(v-R_h v),\nabla v_h \right)=0,\quad
\textrm{for all }
 v_h\in V_h.\label{R_h}
\end{eqnarray}
By classical interpolation theory and finite element theories \cite{Th}, we have
\begin{align}
&
\|v-\Pi_h v\|_{L^2}+h\|\nabla(v-\Pi_h v)\|_{L^2}
\leq Ch^{s+1}\|v\|_{H^{s+1}},
\label{Pih}\\
&
\|v-R_h v\|_{L^2}+h\|\nabla (v-R_h v)\|_{L^2}\leq Ch^{s+1}\|v\|_{H^{s+1}},
\label{Rh}
\end{align}
for any $v\in H_0^1(\Omega)\cap H^{s+1}(\Omega)$ and $1\leq s\leq r$.

From (\ref{i1}), we can see that the exact
solution $u^n$ satisfies the following equation
\begin{align}
 D_{\tau}^{\alpha} u^n-\Delta  u^{n }
=f( u^{n-1})+T^{n}
\label{true1}
\end{align}
with the truncation error $T^{n}$ given by
\[
T^{n}= D_{\tau}^{\alpha} u^n-~_0^C\!\mathcal{D}_{t_n}^\alpha u+ f( u^{n})-f( u^{n-1}).
\]
By \eqref{Qn} and Taylor expansion, we have
\begin{align}\label{loc-301}
\| T^n \|_{L^2} \leq C\tau.
\end{align}

Let
$$
e_h^n=R_hu^n-U_h^n,\quad n=0,1,\cdots,N.
$$
Subtracting \eqref{true1} from the numerical scheme \eqref{s1},
it is easy to see that
$e_h^n$ satisfies
\begin{align}
(D_\tau^\alpha e_h^n,v_h)+(\nabla e_h^n,\nabla v_h)
=(D_\tau^\alpha (R_hu^n-u^n),v_h)
+(f(u^{n-1})-f(U_h^{n-1}),v_h)+(T^n,v_h)
\label{eh}
\end{align}
for any $v_h\in V_h$ and $n=1,2,\cdots,N$.


Taking $v_h=e_h^n$ in (\ref{eh}), we have
\begin{align}
\label{dis-1}
&(D_\tau^\alpha e_h^n,e_h^n)+\|\nabla e_h^n\|_{L^2}^2 \nonumber\\
&\leq
(\frac{L}{2}+1)\|e_h^n\|_{L^2}^2+\frac{L}{2}\|e_h^{n-1}\|_{L^2}^2
+\frac{1}{2}\|D_{\tau}^{\alpha}(R_hu^n-u^n)\|_{L^2}^2+Ch^{2(r+1)}
+\frac{1}{2}\|T^n\|_{L^2}^2\nonumber\\
&\leq
(\frac{L}{2}+1)\|e_h^n\|_{L^2}^2+\frac{L}{2}\|e_h^{n-1}\|_{L^2}^2
+C(\tau+h^{r+1})^2,
\end{align}
where we have used \eqref{loc-301} and
\[
\|D_\tau^{\alpha} R_hu^n \!-\!~_0^C\!\mathcal{D}_{t_n}^\alpha u\|_{L^2}
\leq \|D_\tau^{\alpha} R_hu^n \!-\!~_0^C\!\mathcal{D}_{t_n}^\alpha R_hu\|_{L^2}
+\|~_0^C\!\mathcal{D}_{t_n}^\alpha R_hu -~_0^C\!\mathcal{D}_{t_n}^\alpha u\|_{L^2}
\leq C\tau^{2-\alpha}+Ch^{r+1}.
\]

On the other hand, noting that the coefficients $a_j$ ($j=0,\cdots,N$) defined in \eqref{a} satisfy
\begin{eqnarray*}
1=a_0>a_1>\cdots>a_N>0,
\end{eqnarray*}
we obtain
\begin{align}
\label{dis-2}
(D_\tau^\alpha e^n_h, e_h^n)
=&\frac{\tau^{-\alpha}}{\Gamma(2-\alpha)}
\Big(a_0e_h^n-\sum_{j=1}^{n-1}
(a_{n-j-1} -a_{n-j})e_h^j-a_{n-1}e_h^0,e_h^n\Big)\nonumber\\
 \geq& \frac{\tau^{-\alpha}}{\Gamma(2-\alpha)}
\Big(a_0\|e_h^n\|^2_{L_2}\!-\!\sum_{j=1}^{n-1}
(a_{n\!-\!j-\!1}\!-\!a_{n-j})
\frac{\|e_h^j\|^2_{L_2}\!+\!\|e_h^n\|^2_{L_2}}{2}\!\nonumber\\
&\qquad  -a_{n-1}\frac{\|e_h^0\|^2_{L_2}+\|e_h^n\|^2_{L_2}}{2}\Big)
\nonumber\\
 =&\frac{\tau^{-\alpha}}{2\Gamma(2-\alpha)}
\Big(a_0\|e_h^n\|_{L^2}^2-\sum_{j=1}^{n-1}
(a_{n-j-1} -a_{n-j})\|e_h^j\|_{L^2}^2-a_{n-1}\|e_h^0\|_{L_2}^2\Big)
\nonumber\\
 =&\frac{\tau^{-\alpha}}{2\Gamma(2-\alpha)}
\sum_{j=0}^n b_{n-j} \|e_h^j\|_{L_2}^2
\nonumber\\
=&\frac{1}{2}D_\tau^\alpha \|e_h^n\|^2_{L^2}.
\end{align}
Combining \eqref{dis-1} and \eqref{dis-2}, we get
\[
D_\tau^\alpha \|e_h^n\|^2_{L^2} \leq (L+2)\|e_h^n\|^2_{L^2}+L\|e_h^{n-1}\|^2_{L^2}
+2C(\tau+h^{r+1})^2.
\]
By Lemma \ref{lemma:Gronwall-uniformInitial},
there exists a positive constant $\tau^*$ such that, when $\tau\le\tau^*$,
\begin{align*}
\|e_h^n\|_{L^2}\leq C(\tau+h^{r+1}).
\end{align*}
With \eqref{Rh}, the above estimate further shows that
\begin{align}
\|u^n-U_h^n\|_{L^2}
\leq
\|u^n-R_h u^n\|_{L^2}+\|e_h^n\|_{L^2}
\leq
C(\tau+h^{r+1}).
\end{align}
Taking $\tau_0\leq \tau^*$ and $C_0\ge C$,
the proof of Theorem \ref{main} is complete.
\quad \endproof

\subsection{The proof of Lemma \ref{lemma:Gronwall-uniformInitial}}
To prove Lemma \ref{lemma:Gronwall-uniformInitial},
we first present two useful lemmas.

\begin{lemma}\label{lemma:recursionCoefficient}
Let $\{ p_n \}$ be a sequence defined by
\begin{align}
p_{0}=1,\quad p_{n}=\sum_{j=1}^{n}(a_{j-1}-a_j)p_{n-j},\quad n\geq1.
\label{p}
\end{align}
Then it holds that
\begin{align}
& \text{(i) } \quad 0<p_n < 1, \qquad  \sum_{j=k}^{n}p_{n-j}a_{j-k}=1,
\quad 1\leq k\leq n,
\label{Lem-1}
\\
&\text{(ii) } \quad \Gamma(2-\alpha)\sum_{j=1}^{n}p_{n-j}\leq\frac{n^{\alpha}}{\Gamma(1+\alpha)},
\label{Lem-2}
\end{align}
and for $m =1,2,\cdots$,
\begin{align}
\text{(iii) } \quad\frac{\Gamma(2-\alpha)}{\Gamma(1+(m-1)\alpha)}\sum_{j=1}^{n-1}p_{n-j}j^{(m-1)\alpha}\leq\frac{n^{m\alpha}}{\Gamma(1+m\alpha)}.
\label{Lem-3}
\end{align}
\end{lemma}

\begin{proof}
(i) Since $a_{j-1}>a_j$ for $j\geq1$, it is easy to verify inductively from
\refe{p} that $0< p_n<1$ $(n\ge 1)$. Moreover, we have
\begin{align*}
\Phi_n\equiv&\,\sum_{j=1}^{n}p_{n-j}a_{j-1}=\sum_{j=0}^{n}p_{n-j}a_j
=\sum_{j=1}^{n+1}p_{n+1-j}a_{j-1}=\Phi_{n+1},\quad n\geq1.
\end{align*}
This implies $\Phi_n=\Phi_1=a_{0}p_{0}=1$ for $n\geq1$.
Substituting $j=l+k-1$, we further find
\begin{align*}
\sum_{j=k}^{n}p_{n-j}a_{j-k}=\sum_{l=1}^{n-k+1}p_{n-k+1-l}a_{l-1}=\Phi_{n-k+1}
=\Phi_{n}=1,\quad 1\leq k\leq n.
\end{align*}
The equality (\ref{Lem-1}) is proved.

(ii) To prove (\ref{Lem-2}) and (\ref{Lem-3}),  we introduce an
auxiliary function $q(t)=t^{m\alpha}/\Gamma(1+m\alpha)$ for $m\ge 1$. Then for $j\geq1$, we have
\begin{align}
\label{recursionCoefficient-proof-q}
\int_0^j\frac{(j-s)^{-\alpha}q'(s)}{\Gamma(1-\alpha)}d s
=\frac{B(m\alpha,1-\alpha)j^{(m-1)\alpha}}{\Gamma(1-\alpha)
\Gamma(m\alpha)}
=\frac{j^{(m-1)\alpha}}{\Gamma(1+(m-1)\alpha)},
\end{align}
where we have used the fact that for $z,w>0$
\[
B(z,w) \equiv\int_0^1s^{z-1}(1-s)^{w-1}d s=\frac{\Gamma(z)\Gamma(w)}{\Gamma(z+w)} .
\]

Let $Q(t)$ be a piecewise linear
interpolating polynomial of $q(t)$ satisfying $Q(k)=q^k:=q(k)$.
Moreover for $j\geq1$, we define the approximation error by
\begin{align}\label{app-err}
\int_0^j\frac{q'(s)-Q'(s)}{\Gamma(1-\alpha)(j-s)^{\alpha}}d s
=\sum_{k=1}^j\int_{k-1}^{k}\frac{q'(s)-Q'(s)}{\Gamma(1-\alpha)(j-s)^{\alpha}}d s\
:=
\sum_{k=1}^jR_k^j,
\end{align}
where
\begin{align*}
R_k^j=\int_{k-1}^{k}\frac{d [q(s)-Q(s)]}{\Gamma(1-\alpha)(j-s)^{\alpha}}
=-\frac{\alpha}{\Gamma(1-\alpha)}\int_{k-1}^{k}\frac{q(s)-Q(s)}{(j-s)^{\alpha+1}}d s,
\quad 1\leq k\leq j.
\end{align*}
Combining \eqref{recursionCoefficient-proof-q} and \eqref{app-err} yields
 \begin{eqnarray}\label{recursionCoefficient-proof-globalError}
\frac{j^{(m-1)\alpha}}{\Gamma(1+(m-1)\alpha)}&=&\frac{1}{\Gamma(1-\alpha)}\sum_{k=1}^j\int_{k-1}^{k}\frac{Q'(s)}{(j-s)^{\alpha}}d s+\sum_{k=1}^jR_k^j\nonumber\\
&=&\sum_{k=1}^ja_{j-k}\frac{\delta_tq^k}{\Gamma(2-\alpha)}+\sum_{k=1}^jR_k^j.
\end{eqnarray}

Noting that $q(t)$ is concave (i.e., $q''(t)\le 0$) for $m=1$, we have $Q(t) \le q(t)$,
$R_k^j \le 0$ and
\begin{eqnarray}
\label{recurs-1}
1 \leq \sum_{k=1}^ja_{j-k}\frac{\delta_tq^k}{\Gamma(2-\alpha)}.
\end{eqnarray}
Multiplying \eqref{recurs-1} by $\Gamma(2-\alpha)p_{n-j}$
and summing it over for $j$ from $1$ to $n$, we have
\begin{eqnarray*}
\Gamma(2\!-\!\alpha) \sum_{j=1}^{n}p_{n\!-\!j}
\leq\sum_{j=1}^{n}p_{n-j}\sum_{k=1}^{j}a_{j-k}{\delta_tq^k}
\!=\!\sum_{k=1}^{n}\delta_tq^k\sum_{j=k}^{n}p_{n-j}a_{j-k}
\!=\! \sum_{k=1}^{n}\delta_tq^k
\!=\! \frac{n^{\alpha}}{\Gamma(1+\alpha)},
\end{eqnarray*}
where we have used the equality (\ref{Lem-1}).

(iii) We multiply \eqref{recursionCoefficient-proof-globalError} by $\Gamma(2-\alpha)p_{n-j}$
and sum the resulting equality for $j$ from $1$ to $n-1$ to obtain
\begin{eqnarray}\label{recurs-201}
\frac{\Gamma(2-\alpha)}{\Gamma(1+(m-1)\alpha)}
\sum_{j=1}^{n-1}p_{n-j}j^{(m-1)\alpha}
&=&\sum_{j=1}^{n-1}p_{n-j}\sum_{k=1}^{j}a_{j-k}{\delta_tq^k}
+\Gamma(2-\alpha)\sum_{j=1}^{n-1}p_{n-j}\sum_{k=1}^jR_k^j\nonumber\\
&=&\sum_{k=1}^{n-1}\delta_tq^k\sum_{j=k}^{n-1}p_{n-j}a_{j-k}
+\Gamma(2-\alpha)\sum_{j=1}^{n-1}p_{n-j}\sum_{k=1}^jR_k^j\nonumber\\
&\leq&\sum_{k=1}^{n-1}\delta_tq^k+\Gamma(2-\alpha)\sum_{j=1}^{n-1}p_{n-j}\sum_{k=1}^jR_k^j\nonumber\\
&=&\frac{(n-1)^{m\alpha}}{\Gamma(1+m\alpha)}
+\Gamma(2-\alpha)\sum_{j=1}^{n-1}p_{n-j}\sum_{k=1}^jR_k^j.
\end{eqnarray}

If $1\leq m\leq1/\alpha$, $q(t)$ is still concave (i.e., $q''(t) \le 0$). Then
$R_k^j\leq0$ and
(\ref{Lem-3}) follows immediately from the above estimate.

If $m>1/\alpha$, by \refe{app-err}, we have
\begin{eqnarray}
\label{recurs-21}
R_k^j
&=&\,\int_{k-1}^k\frac{(j-s)^{-\alpha}}{\Gamma(1-\alpha)}
\int_{k-1}^{k}(q'(s)-q'(\mu))d\mu d s\nonumber\\
&=&\int_{k-1}^k\frac{(j-s)^{-\alpha}}{\Gamma(1-\alpha)}\int_{k-1}^{k}
\int_{\mu}^{s}q''(\eta)d\eta d\mu ds
\nonumber\\
&\leq&\,\int_{k-1}^k\frac{(j-s)^{-\alpha}}{\Gamma(1-\alpha)}\int_{k-1}^{k}\int_{\mu}^{k}
\frac{d\eta^{m\alpha-1}}{\Gamma(m\alpha)}d\mu ds\nonumber\\
&=&a_{j-k}\int_{k-1}^{k}\frac{k^{m\alpha-1}-\mu^{m\alpha-1}}{\Gamma(2-\alpha)
\Gamma(m\alpha)}d\mu,\quad 1\leq k\leq j.
\end{eqnarray}
Therefore, by applying \eqref{Lem-1} for $ n\geq1$, we have
\begin{eqnarray}\label{recurs-22}
\Gamma(2-\alpha)\sum_{j=1}^{n-1}p_{n-j}\sum_{k=1}^jR_k^j
&\leq&\sum_{j=1}^{n-1}p_{n-j}\sum_{k=1}^ja_{j-k}\int_{k-1}^{k}
\frac{k^{m\alpha-1}-\mu^{m\alpha-1}}{\Gamma(m\alpha)}d\mu\nonumber\\
&=&\sum_{k=1}^{n-1}\int_{k-1}^{k}\frac{k^{m\alpha-1}-\mu^{m\alpha-1}}{\Gamma(m\alpha)}d\mu\sum_{j=k}^{n-1}p_{n-j}a_{j-k}\nonumber\\
&\leq&\sum_{k=1}^{n-1}\frac{k^{m\alpha-1}}{\Gamma(m\alpha)}
-\frac{(n-1)^{m\alpha}}{\Gamma(1+m\alpha)}\nonumber\\
&\leq&\frac{n^{m\alpha}}{\Gamma(1+m\alpha)}
-\frac{(n-1)^{m\alpha}}{\Gamma(1+m\alpha)}.
\end{eqnarray}
Substituting \eqref{recurs-22} into \eqref{recurs-201}, the proof of (\ref{Lem-3}) is complete.
\end{proof}

\begin{lemma}\label{lemma:2}
Let $\overrightarrow{e}=(1,1,\cdots,1)^T \in R^n$ and
\begin{equation}\label{matr-0}
 J=2\Gamma(2-\alpha)\lambda\tau^\alpha\left[\begin{matrix}
    0 &~ p_1   &~\cdots &p_{n-2}  & p_{n-1} \\
    0&~  0    &~\cdots& ~p_{n-3}& ~p_{n-2}\\
    \vdots&~ \vdots    &~ \ddots & ~\vdots& ~\vdots\\
    0 & ~0  &~\cdots&~ 0&~p_1 \\
    0& ~0& ~\cdots& ~0& ~0\\
\end{matrix}\right]_{n\times n}.
\end{equation}
Then, it holds that
\begin{itemize}
\item [{(i)}] $J^i=0,~~i\geq n$;
\item [{(ii)}]
$J^m\overrightarrow{e} \leq \frac{1}{\Gamma(1+m\alpha)}
\Big( (2\lambda t_n^\alpha)^m, (2\lambda t_{n-1}^\alpha)^m,
\cdots,(2\lambda t_1^\alpha)^m\Big)^T$,~~ $m=0,1,2,\cdots$;
\item [{(iii)}]
$\sum\limits_{j=0}^iJ^j\overrightarrow{e} = \sum\limits_{j=0}^{n-1}J^{j}\overrightarrow{e}
\leq\Big( E_\alpha(2\lambda t_n^\alpha), E_\alpha(2\lambda t_{n-1}^\alpha) ,
\cdots,E_\alpha(2\lambda t_1^\alpha)\Big)^T$,~~ $i \geq n$.
\end{itemize}
\end{lemma}
\begin{proof}
 Noting that $J$ is an upper triangular matrix, it is easy to check that (i) holds.

To prove (ii), we apply the mathematical induction.
It is obvious that (ii) holds for $m=0$. We assume that (ii) holds for $m=k$.
Since $t_n=n \tau$ and \refe{matr-0}, we have
\begin{eqnarray}
J^{k+1}\overrightarrow{e}&=& J J^{k}\overrightarrow{e} \leq \frac{1}{ \Gamma(1+k\alpha)} J \Big( (2\lambda t_n^\alpha)^k,
(2\lambda t_{n-1}^\alpha)^k,
\cdots,(2\lambda t_1^\alpha)^k\Big)^T\\
&=&
\frac{\Gamma(2-\alpha)(2\lambda\tau^\alpha)^{k+1}  }{\Gamma(1+k\alpha)}
\Big(  \sum_{j=1}^{n-1}p_{n-j}j^{k\alpha},  \sum_{j=1}^{n-2}p_{n-1-j} (j-1)^{k\alpha},
\cdots, p_11^{k\alpha}, 0\Big )^T.
\end{eqnarray}
By using \eqref{Lem-3} in Lemma \ref{lemma:recursionCoefficient}, we further have
\begin{eqnarray}
J^{k+1}\overrightarrow{e}
&\leq&\frac{ (2\lambda\tau^\alpha)^{k+1}  }{\Gamma(1+(k+1)\alpha)}
\Big( n^{(k+1)\alpha},  (n-1)^{(k+1)\alpha},\cdots, 2^{(k+1)\alpha},    1^{(k+1)\alpha}\Big )^T
\nonumber\\
&=& \frac{1}{\Gamma(1+(k+1)\alpha)}
\Big( (2\lambda t_n^\alpha)^{k+1}, (2\lambda t_{n-1}^\alpha)^{k+1},
\cdots,(2\lambda t_1^\alpha)^{k+1}\Big)^T .
\end{eqnarray}
Thus (ii) holds for $m=k+1$.

Since (i) implies that
$\sum_{j=0}^iJ^j\overrightarrow{e}
 = \sum_{j=0}^{n-1}J^{j}\overrightarrow{e}$ for $i\geq n$, and by (ii), we have
\begin{eqnarray}
\sum_{j=0}^{n-1}J^{j}\overrightarrow{e}
&\leq& \sum_{j=0}^{n-1} \frac{1}{\Gamma(1+j\alpha)}
\Big( (2\lambda t_n^\alpha)^{j}, (2\lambda t_{n-1}^\alpha)^{j},
\cdots,(2\lambda t_1^\alpha)^{j}\Big)^T \nonumber\\
&\leq & \Big( E_\alpha(2\lambda t_n^\alpha), E_\alpha(2\lambda t_{n-1}^\alpha) ,
\cdots,E_\alpha(2\lambda t_1^\alpha)\Big)^T.
\end{eqnarray}
The proof of Lemma \ref{lemma:2} is complete.
\end{proof}

\vspace{0.1in}
We now turn back to the proof of Lemma \ref{lemma:Gronwall-uniformInitial}.

By the definition of $L1$-approximation \refe{l1}, we get
\begin{equation} \label{GWN}
\sum_{k=1}^ja_{j-k}{\delta_t\omega^k}\leq \Gamma(2-\alpha)\tau^{\alpha}
(\lambda_1\omega^j+\lambda_2\omega^{j-1})
+\Gamma(2-\alpha)\tau^{\alpha}g^j.
\end{equation}
Multiplying the inequality \eqref{GWN} by $p_{n-j}$ and summing over
for $j$ from $1$ to $n$, we have
\begin{align*}
\sum_{j=1}^np_{n\!-\!j}\sum_{k=1}^ja_{j\!-\!k}{\delta_t\omega^k}\!
\leq \! \Gamma(2\!-\!\alpha)\tau^{\alpha}\sum_{j=1}^np_{n\!-\!j}(\lambda_1\omega^j+\lambda_2\omega^{j\!-\!1})
+\Gamma(2\!-\!\alpha)\tau^{\alpha}\sum_{j=1}^{n}p_{n\!-\!j}g^j.
\end{align*}
By using the results \eqref{Lem-1} and \eqref{Lem-2} in Lemma \ref{lemma:recursionCoefficient},
we obtain
\begin{align*}
\sum_{j=1}^np_{n-j}\sum_{k=1}^ja_{j-k}{\delta_t\omega^k}=\sum_{k=1}^n\delta_t\omega^k\sum_{j=k}^np_{n-j}a_{j-k}
=\sum_{k=1}^n\delta_t\omega^k=\omega^n-\omega^0, \quad n\geq1,
\end{align*}
and
\begin{align*}
\Gamma(2-\alpha)\tau^{\alpha}\sum_{j=1}^{n}p_{n-j}g^j
\leq \Gamma(2-\alpha)\tau^{\alpha}\max_{1\leq j\leq n}g^j
\sum_{j=1}^{n}p_{n-j}\leq \frac{t_n^{\alpha}}{\Gamma(1+\alpha)}
\max_{1\leq j\leq n}g^j,\quad n\geq1.
\end{align*}
It follows that
\begin{align*}
\omega^n\leq \Psi_n+\Gamma(2-\alpha)\tau^{\alpha}\sum_{j=1}^{n}p_{n-j}(\lambda_1\omega^j+\lambda_2\omega^{j-1})~~,\quad
n\geq1,
\end{align*}
where
\[
\Psi_n : =\omega^0+\frac{t_n^{\alpha}}{\Gamma(1+\alpha)}\max_{1\leq j\leq n}g^j.
\]
By noting that $\Psi_n\geq \Psi_k$ for $n\geq k\geq1$, we get
\begin{align}\label{ImmediateInequality2-uniformInitial}
\omega^n\leq 2\Psi_n+2\Gamma(2-\alpha)\Big(\lambda_1\tau^{\alpha}\sum_{j=1}^{n-1}p_{n-j}\omega^j+\lambda_2\tau^{\alpha}\sum_{j=1}^{n}
p_{n-j}\omega^{j-1}\Big)~~,\quad
n\geq1,
\end{align}
when $\tau\leq \sqrt[\alpha]{\frac{1}{2\Gamma(2-\alpha)\lambda_1}}$.

Let $V=(\omega^n,\omega^{n-1},\cdots,\omega^1)^T$.
Thus \eqref{ImmediateInequality2-uniformInitial} can be written in a vector form by
\begin{equation}\label{matr-1}
V \leq(\lambda_1J_1+\lambda_2J_2) V + 2\Psi_n\overrightarrow{e},
\end{equation}
where
\begin{equation*}
 J_1= 2\Gamma(2-\alpha)\tau^\alpha\left[\begin{matrix}
    0 &~ p_1   &~\cdots &p_{n-2}  & p_{n-1} \\
    0&~  0    &~\cdots& ~p_{n-3}& ~p_{n-2}\\
    \vdots&~ \vdots    &~ \ddots & ~\vdots& ~\vdots\\
    0 & ~0  &~\cdots&~ 0&~p_1 \\
    0& ~0& ~\cdots& ~0& ~0\\
\end{matrix}\right]_{n\times n},~~
\end{equation*}
and
\begin{equation*}
 J_2=2\Gamma(2-\alpha)\tau^\alpha\left[\begin{matrix}
    0 &~ p_0   &~\cdots &p_{n-3}  & p_{n-2} \\
    0&~  0    &~\cdots& ~p_{n-4}& ~p_{n-3}\\
    \vdots&~ \vdots    &~ \ddots & ~\vdots& ~\vdots\\
    0 & ~0  &~\cdots&~ 0&~p_0 \\
    0& ~0& ~\cdots& ~0& ~0\\
\end{matrix}\right]_{n\times n}.
\end{equation*}
By \refe{p}, we have
\begin{equation*}
p_i \leq \frac{1}{a_0-a_1}p_{i+1} = \frac{1}{ 2-2^{1-\alpha} } p_{i+1},~~ i\geq 0.
\end{equation*}
Therefore,
\begin{equation}\label{matr-2}
J_2V \leq \frac{1}{ 2-2^{1-\alpha}}J_1V.
\end{equation}
Substituting \eqref{matr-2} into \eqref{matr-1},  we get
\begin{equation}\label{mart-3}
V \leq J V + 2\Psi_n\overrightarrow{e},
\end{equation}
where $J$ is defined in \eqref{matr-0} with $\lambda=\lambda_1+\frac{\lambda_2}{2-2^{1-\alpha}} $.

As a result, we see that
\begin{align}
\label{gron-ineq}
 V &\leq J V + 2\Psi_n\overrightarrow{e}
\nonumber \\
& \leq J(JV+2\Psi_n\overrightarrow{e})+2\Psi_n\overrightarrow{e}\nonumber \\
&=J^2V
+2\Psi_n\sum_{j=0}^1J^j\overrightarrow{e}
\nonumber\\
 &\leq\cdots \nonumber \\
 &\leq J^nV+ 2\Psi_n\sum_{j=0}^{n-1}J^j\overrightarrow{e}.
 \end{align}
By using (i) and (iii) in Lemma \ref{lemma:2}, we obtain \eqref{gron-101} and complete the proof of Lemma \ref{lemma:Gronwall-uniformInitial}.
\quad \endproof

\section{Numerical examples}\label{sec4}
\setcounter{equation}{0}
In this section, we present three numerical examples which
substantiate the analysis given earlier for schemes \eqref{s1}, \eqref{s2}
and \eqref{s3}. The orders of convergence are
examined. The exact solutions
of equations in the first two examples are smooth and the computations
are performed by using the software FreeFEM++.  The exact solution of the equation in last example has an initial singularity and the computation is performed by using Matlab.

\vspace{0.1in}
{\bf Example 1.} We first consider the two-dimensional time-fractional Huxley equation
\begin{eqnarray}\label{exam-2}
 & ~_0^C\!\mathcal{D}_t^\alpha u= \Delta u + u(1-u)(u-1) + g_1,
\qquad x \in [0,1]\times[0,1],~~
0 < t \le 1.
\end{eqnarray}
The equation \eqref{exam-2} can describe many different physical models, such as population genetics
in circuit theory and the transmission of nerve impulses
\cite{merd12,li-zhang16}. To obtain a simple benchmark
solution, we can calculate the function $g_1$ based on the exact solution
\[
 u = (1+t^3)(1-x)\sin(x)(1-y)\sin(y).
\]

\begin{table}[!ht]
\begin{center}
\caption{ $L^2$-errors $\|u^N-U_h^N\|_{L^2}$ and convergence rates in temporal direction  for Eq. \eqref{exam-2}}
\begin{tabular}{llllllllll}
\hline
& &\multicolumn{2}{c}{$\alpha=0.25$}&{}&\multicolumn{2}{c}{$\alpha=0.5$}&{}&\multicolumn{2}{c}{$\alpha=0.75$}\\
\cline{3-4}\cline{6-7}\cline{9-10}
&{$N$}&$\mbox{error}$&$\mbox{order}$&{}&$\mbox{error}$&$\mbox{order}$&{}&$\mbox{error}$&$\mbox{order}$\\
\hline
                  &$10$        &2.81E-4   &--      &{}    &3.19E-4      &--        &{}    &4.20E-4   &--\\
Scheme \eqref{s1} &$20$        &1.43E-4   &0.96   &{}    &1.57E-4      &1.02     &{}    &2.04E-4   &1.04\\
                  &$40$        &7.20E-5   &0.99   &{}    &7.72E-5      &1.02     &{}    &9.95E-5   &1.05\\
                  &$80$        &3.60E-5   &1.00   &{}    &3.79E-5      &1.02     &{}    &4.73E-5   &1.05\\
\hline
                  &$10$        &6.42E-6   &--      &{}    &1.06E-4      &--        &{}    &1.50E-4   &--\\
Scheme \eqref{s2} &$20$        &2.46E-6   &1.38   &{}    &3.37E-5      &1.37     &{}    &6.59E-5   &1.18\\
                  &$40$        &8.99E-7   &1.45   &{}    &6.75E-6      &1.41     &{}    &2.85E-5   &1.21\\
                  &$80$        &3.17E-7   &1.53   &{}    &2.49E-6      &1.44     &{}    &1.22E-5   &1.23\\
\hline
                  &$10$        &6.62E-5   &--      &{}    &1.06E-4      &--        &{}    &2.09E-4   &--\\
Scheme \eqref{s3} &$20$        &1.83E-5   &1.85   &{}    &3.37E-5      &1.65     &{}    &8.17E-5   &1.36\\
                  &$40$        &4.97E-6   &1.88   &{}    &1.08E-5      &1.64     &{}    &3.25E-5   &1.32\\
                  &$80$        &1.35E-6   &1.88   &{}    &3.53E-6      &1.62     &{}    &1.32E-5   &1.30\\
\hline
\end{tabular}\label{table1}
\end{center}
\end{table}

\begin{table}[!ht]
\begin{center}
\caption{$L^2$-errors $\|u^N-U_h^N\|_{L^2}$ and convergence rates in spatial  direction for Eq. \eqref{exam-2}}
\begin{tabular}{llllllllll}
\hline
 &\multicolumn{2}{c}{$\mbox{L-FEM}$}& &{}&\multicolumn{2}{c}{$\mbox{Q-FEM}$}&{}\\
\cline{2-3}\cline{6-7}
{$M$}&$\mbox{error}$&$\mbox{order}$& &{}&$\mbox{error}$&$\mbox{order}$&{}\\
\hline
$5$          &6.16E-3   &--         &{}  &   &2.08E-4      &--        &{}    \\
$10$         &1.57E-3   &1.97      &{}  &   &2.61E-5      &2.99        &{}    \\
$20$         &3.96E-4   &1.99      &{}  &   &3.26E-6     &3.01     &{}    \\
$40$         &9.91E-5   &2.00      &{}  &   &4.08E-7      &3.00     &{}    \\
\hline
\end{tabular}\label{table2}
\end{center}
\end{table}
We apply the linearized schemes \eqref{s1}, \refe{s2} and \eqref{s3}
to solve problem \eqref{exam-2} with linear and quadratic
finite element approximations, respectively. Here and below,
a uniform triangular partition with $M+1$ nodes in each spatial direction
is used. To investigate the temporal convergence order, we
use a quadratic FEM with a fixed spatial meshsize $h=1/100$
and several refined temporal meshes $\tau$.
Table \ref{table1} shows the $L^2$-errors at time $T=1$ and
convergence rates in temporal direction with different $\alpha$.
From Table \ref{table1}, one can see that the numerical schemes \eqref{s2}
and \eqref{s3} have an accuracy
of order $2\!-\!\alpha$, while numerical scheme \eqref{s1} has an
accuracy of order $1$.

To investigate spatial convergence order, we
apply the scheme \eqref{s1} to
solve equation \eqref{exam-2}
using both linear and quadratic FEMs with several refined spatial meshes $h$. Table \ref{table2} shows
the $L^2$-errors and convergence rates with $\alpha=0.25$ and $N=M^3$.
The results in Table \ref{table2} indicate that
the scheme \eqref{s1}
is of optimal convergence order ${r+1}$ in spatial direction.

\vspace{0.1in}
{\bf Example 2.}
Secondly, we consider the three-dimensional time-fractional Fisher equation
\begin{eqnarray}\label{exam-3}
&& ~_0^C\!\mathcal{D}_t^\alpha u= \Delta u + u(1-u) + g_2, \qquad x \in [0,1]\times[0,1]\times[0,1],\
0 < t \le 1.
\end{eqnarray}
The equation \eqref{exam-3} was originally proposed to describe the spatial and temporal
propagation of a virile gene. Later, it is revised by providing some
characteristics of memory embedded into the
system \cite{alq15,li-zhang16}. To get a benchmark solution,
we calculate the right-hand side $g_2$ of \eqref{exam-3}
based on the exact solution
\[ u = t^2\sin(\pi x)\sin(\pi y)\sin(\pi z). \]

We apply all three proposed schemes with quadratic FEMs  to solve the equation \eqref{exam-3} by taking
$M=60$ and several refined temporal meshes.
Table \ref{table3} shows the $L^2$-errors at time $T=1$ and
convergence rates in temporal direction with different $\alpha$.
Table \ref{table4} shows $L^2$-errors
at time $T=1$ and convergence rates in spatial direction for the scheme \eqref{s1} with $\alpha=0.25$ and $N=M^3$.
Again, the results in Tables \ref{table3} and \ref{table4} confirm our theoretical analysis.

\begin{table}[!ht]
\begin{center}
\caption{ $L^2$-errors $\|u^N-U_h^N\|_{L^2}$ and convergence rates in temporal direction  for Eq. \eqref{exam-3}}
\begin{tabular}{llllllllll}
\hline
& &\multicolumn{2}{c}{$\alpha=0.25$}&{}&\multicolumn{2}{c}{$\alpha=0.5$}&{}&\multicolumn{2}{c}{$\alpha=0.75$}\\
\cline{3-4}\cline{6-7}\cline{9-10}
&{$N$}&$\mbox{error}$&$\mbox{order}$&{}&$\mbox{error}$&$\mbox{order}$&{}&$\mbox{error}$&$\mbox{order}$\\
\hline

                  &$5$         &3.48E-4   &--      &{}    &4.16E-4      &--        &{}    &6.50E-4   &--\\
Scheme \eqref{s1} &$10$         &2.24E-4   &0.64   &{}    &2.47E-4      &0.75     &{}    &3.51E-4   &0.89\\
                  &$20$         &1.25E-4   &0.84   &{}    &1.32E-4      &0.90     &{}    &1.76E-4   &0.99\\
                  &$40$         &6.61E-5   &0.92      &{}    &6.75E-5      &0.97     &{}    &8.65E-5   &1.02\\
\hline
                  &$5$         &1.05E-3   &--      &{}    &1.34E-3      &--        &{}    &1.98E-3   &--\\
Scheme \eqref{s2} &$10$        &2.96E-4   &1.82   &{}    &4.11E-4      &1.70     &{}    &7.09E-4   &1.48\\
                  &$20$        &7.99E-5   &1.88   &{}    &1.24E-4      &1.72     &{}    &2.60E-4   &1.45\\
                  &$40$        &2.15E-5   &1.89   &{}    &3.77E-5      &1.72     &{}    &9.89E-5   &1.39\\
\hline
                  &$5$         &3.04E-4   &--      &{}    &5.85E-4      &--        &{}    &1.21E-3   &--\\
Scheme \eqref{s3} &$10$        &9.26E-5   &1.72   &{}    &2.04E-4      &1.52     &{}    &4.99E-4   &1.28\\
                  &$20$        &2.65E-5   &1.80   &{}    &6.91E-5      &1.56     &{}    &2.05E-4   &1.28\\
                  &$40$        &7.86E-6   &1.74   &{}    &2.39E-5      &1.53     &{}    &8.46E-5   &1.27\\
\hline
\end{tabular}\label{table3}
\end{center}
\end{table}

\begin{table}[!ht]
\begin{center}
\caption{$L^2$-errors $\|u^N-U_h^N\|_{L^2}$ and convergence rates in spatial  direction  for Eq. \eqref{exam-3}}
\begin{tabular}{llllllllll}
\hline
 &\multicolumn{2}{c}{$\mbox{L-FEM}$}& &{}&\multicolumn{2}{c}{$\mbox{Q-FEM}$}&{}\\
\cline{2-3}\cline{6-7}
{$M$}&$\mbox{error}$&$\mbox{order}$& &{}&$\mbox{error}$&$\mbox{order}$&{}\\
\hline
$5$          &5.73E-2   &--      &{} &   &2.63E-3      &--        &{}    \\
$10$         &1.54E-2   &1.90   &{} &   &3.27E-4      &3.01     &{}    \\
$20$         &3.91E-3   &1.97   &{} &   &4.09E-5      &3.00     &{}    \\
$40$         &9.86E-4   &1.99   &{} &   &5.11E-6      &3.00     &{}    \\
\hline
\end{tabular}\label{table4}
\end{center}
\end{table}

\vspace{0.1in}
{\bf Example 3.} We finally consider the time-fractional
Fokker-Planck equation
\begin{eqnarray}\label{exam-FP}
&& ~_0^C\!\mathcal{D}_t^\alpha u= u_{xx} + \frac{\phi'(x)}{\eta_\alpha}u_x +\frac{\phi''(x)}{\eta_\alpha}u + g_3, \qquad x \in [0,\pi],~~
0 < t \le 1.
\end{eqnarray}
The model describes the time evolution of the probability density
function of position and velocity of a particle \cite{bar00,den07}.
Here $u$ is the probability density, $\phi$ indicates the
potential of overdamped Brownian motion, $\eta_\alpha$ is
the generalized friction coefficient. We set $\phi(x)=\exp(x)$,
$\eta_\alpha=1$, calculate the function $g_3$ based on the exact solution
\[ u = (t^\alpha+t^2)\sin(x). \]

The exact solution $u$ has an initial layer at $t=0$ since
the derivative of the solution, i.e., $u_t(x,t)$, blows up as $t\rightarrow 0+$. Clearly, the solution does not have the
requisite regularity.  We solve the linear equation \eqref{exam-FP}
by the proposed three schemes with linear finite element
approximation on uniform meshes. We set $h=10^{-4}$ and investigate the temporal convergence order by refining the temporal mesh $\tau$.
The errors {$\max_{1\leq n \leq N} \|u^n-U^n_h\|_{L^2}$} and convergence rates in the temporal direction with different $\alpha$
are listed in Table \ref{table5}.
The results in Table \ref{table5} indicate that schemes
\eqref{s1}, \eqref{s2} and \eqref{s3} are convergent, but the convergence rate is not of order $1$ or $2-\alpha$ in the temporal direction any more. These results agree with the theoretical
result shown in Remark 1.
\begin{table}[!ht]
\begin{center}
\caption{The errors $\max\limits_{1\leq n \leq N} \|u^n-U^n_h\|_{L^2}$ and convergence rates in temporal direction for Eq. \eqref{exam-FP}}
\begin{tabular}{llllllllll}
\hline
& &\multicolumn{2}{c}{$\alpha=0.4$}&{}&\multicolumn{2}{c}{$\alpha=0.6$}&{}&\multicolumn{2}{c}{$\alpha=0.8$}\\
\cline{3-4}\cline{6-7}\cline{9-10}
&{$N$}&$\mbox{error}$&$\mbox{order}$&{}&$\mbox{error}$&$\mbox{order}$&{}&$\mbox{error}$&$\mbox{order}$\\
\hline
                  &$50$         &1.91E-1   &--      &{}    &2.08E-1      &--        &{}    &2.21E-1   &--\\
                  &$100$        &1.13E-1   &0.75   &{}    &1.06E-1      &0.97     &{}    &1.13E-1   &0.96\\
Scheme \eqref{s1} &$200$        &7.63E-2   &0.58   &{}    &5.36E-2      &0.98     &{}    &5.73E-2   &0.98\\
                  &$400$        &5.07E-2   &0.58   &{}    &2.69E-2      &0.99     &{}    &2.89E-2   &0.99\\
                  &$800$        &3.36E-2   &0.59  &{}     &1.35E-2      &0.99     &{}    &1.45E-2   &0.99\\
\hline
                  &$50$         &4.57E-2   &--      &{}    &2.21E-2      &--        &{}    &7.57E-3   &--\\
                  &$100$        &3.59E-2   &0.35   &{}    &1.47E-2      &0.59     &{}    &4.59E-3   &0.72\\
Scheme \eqref{s2} &$200$        &2.78E-2   &0.37   &{}    &9.55E-3      &0.62     &{}    &2.67E-3   &0.78\\
                  &$400$        &2.13E-2   &0.39   &{}    &6.17E-3      &0.63     &{}    &1.50E-3   &0.83\\
                  &$800$        &1.61E-2   &0.40   &{}    &3.98E-3      &0.63     &{}    &8.25E-4   &0.86\\
\hline
                  &$50$         &1.38E-1   &--      &{}    &6.48E-2      &--        &{}    &3.53E-2   &--\\
                  &$100$        &1.06E-1   &0.38   &{}    &4.17E-2      &0.64     &{}    &1.93E-2   &0.87\\
Scheme \eqref{s3} &$200$        &8.07E-2   &0.39   &{}    &2.67E-2      &0.64     &{}    &1.07E-2   &0.85\\
                  &$400$        &6.08E-2   &0.41   &{}    &1.71E-2      &0.64     &{}    &6.04E-3   &0.83\\
                  &$800$        &4.56E-2   &0.41   &{}    &1.11E-2      &0.63     &{}    &3.43E-3   &0.82\\
\hline
\end{tabular}\label{table5}
\end{center}
\end{table}

\section{Conclusions}\label{sec5}
Several linearized $L1$-Galerkin FEMs have been proposed for
solving time-fractional nonlinear parabolic PDEs \eqref{i1}
to avoid the iterations at each time step.
Error estimates in previous literatures were
generally obtained only in a small (local) time interval or
in the case that the evolution of
the numerical solution decreases in time. Clearly, it
limits the applications of L1-type methods. In this paper,
we establish a fundamental Gronwall type inequality for
$L1$ approximation to the Caputo fractional derivative,
and provide theoretical analysis to derive the corresponding
optimal error
estimates without the restrictions required in previous works.
A broad range of numerical examples are given to illustrate our
theoretical results.\\

\section*{Acknowledgements}

The research was supported by NSFC under grants 11571128, 91430216 and U1530401, 11372354
 a grant CityU 11302915 from the Research Grants
Council of the Hong Kong Special Administrative Region, and a grant DRA2015518 from 333 High-level Personal Training Project of Jiangsu Province.

\end{document}